\documentclass[10pt]{amsart}

\usepackage[english]{babel}

\usepackage{amsmath}
\usepackage{amssymb}
\usepackage{amsfonts}
\usepackage{graphicx}

\newtheorem{lemma}{Lemma}
\newtheorem{theorem}{Theorem}
\newtheorem{definition}{Definition}
\newtheorem{corollary}{Corollary}
\newtheorem{proposition}{Remark}
\newtheorem{ob}{Notation}
\newtheorem{prim}{Example}
\newtheorem{question}{Question}

\newcommand{\inftree}{{}^{<\omega}\omega}
\newcommand{\bairespace}{\omega^\omega}
\newcommand{\baireset}{{}^\omega\omega}

\begin{document}
\renewcommand{\refname}{References}

\thispagestyle{empty}

\title[Open images of the Sorgenfrey line]{Open images of the Sorgenfrey line}
\author{{Vlad Smolin}}
\address{Vlad Smolin
\newline\hphantom{iii} Krasovskii Institute of Mathematics and Mechanics,
\newline\hphantom{iii} Sofia Kovalevskaya street, 16,
\newline\hphantom{iii} 620990, Ekaterinburg, Russia}
\email{SVRusl@yandex.ru}

\maketitle{\small
\begin{quote}
\noindent{\sc Abstract. } We give a description of Hausdorff continuous open images of the Sorgenfrey line: these are precisely those spaces that have a Sorgenfrey base. Using this description we prove that no Hausdorff compact space that contains a copy of the Sorgenfrey line is a continuous open image of it; in particular the double-arrow space is not a continuous open image of the Sorgenfrey line. \medskip

 \end{quote}
}

\section{Introduction}

A continuous map is called open if the image of an open set under this map is open.

Results in the paper arose from the questions that were posed to the author by E. G. Pytkeev and M. A. Patrakeev:
\begin{question} \label{PytkeevQues}
    Is the double-arrow (two-arrow) space a continuous open image of the Sorgenfrey line?
\end{question}

\begin{question}
    Suppose that a compact space is a continuous open image of the Sorgenfrey line. What can we say about this space?
\end{question}

Previously, continuous open images of the Sorgenfrey line were studied in the class of metrizable spaces. In \cite{7} S. A. Svetlichnyi proved that if a metrizable space is a continuous open image of the Sorgenfrey line, then it is a Polish space, i.e. separable completely metrizable space. In \cite{5} and \cite{6} N. V. Velichko and M. A. Patrakeev independently constructed a continuous open map from the Sorgenfrey line onto the real line. Velichko also proved that for each such map there exists a point with the preimage of cardinality continuum. In \cite{3} Patrakeev proved that continuous open metrizable images of the Sorgenfrey line are exactly Polish spaces. He also strengthened the result of Velichko by showing that for each continuous open map from the Sorgenfrey line onto the metrizable space there exists a point with the preimage of cardinality continuum.

Continuous open images of submetrizable spaces were studied by Svetlichnyi. In \cite{10} he proved that if a paracompact space is an image of a submetrizable space under continuous open compact map, then it is a submetrizable space. He also proved that there exists a nonmetrizable (hence not submetrizable) compact space that is a continuous open s-image of a submetrizable space. Since the Sorgenfrey line is a hereditarily separable submetrizable space it is natural to ask the following question.

\begin{question}
    Is there exists a Hausdorff nonmetrizable compact space that is a continuous open image of the Sorgenfrey line?
\end{question}

We give a description of Hausdorff open images of the Sorgenfrey line: these are precisely those spaces that have a Sorgenfrey base (see Theorem \ref{desc_op_im}). Using this description we prove that no compact space that contains a homeomorphic copy of the Sorgenfrey line is an open image of it (see Corollary \ref{ifcontsnotim}). This result answers Question \ref{PytkeevQues} negatively. But the following question remains open.

\begin{question}
    Is there exists a Hausdorff compact space that contains an uncountable subspace of the Sorgenfrey line and that is a continuous open image of the Sorgenfrey line?
\end{question}

\section{Notation and terminology}

We use terminology from \cite{1}, \cite{2}, and \cite{4}. Also we use the following notations:

\begin{ob} {\rm The symbol := means ``equals by definition''; the symbol $:\longleftrightarrow$ is used to show that the expression on the left side is an abbreviation for expression on the right side;
    \begin{itemize}
        \item $\omega := $ the set of finite ordinals = the set of natural numbers;
        \item $0 = \emptyset \in \omega$;
        \item $n = \{0, \dots, n - 1\}$ for all $n \in \omega$;
        \item $s$ is a {\it sequence} $\ :\longleftrightarrow\ $ $s$ is a function such that $\mathsf{dom}(s) \in \omega$ or $\mathsf{dom}(s) = \omega$;
        \item if $s$ is a sequence, then $\mathsf{lh}(s) := \mathsf{dom}(s)$;
        \item $\langle s_0, \dots, s_{n-1} \rangle :=$ the sequence $s$ such that $\mathsf{lh}(s) = n \in \omega$ and $s(i)=s_i$ for all $i \in n$;
        \item $\langle \rangle :=$ the sequence $q$ such that $\mathsf{lh}(q) = 0$;
        \item if $s = \langle s_0, \dots, s_{n-1} \rangle$, then $s\ \hat{}\ x := \langle s_0, \dots, s_{n-1}, x \rangle$;
        \item $f{\upharpoonright} A :=$ the restriction of function $f$ to $A$;
        \item if $s$ and $t$ are sequences, then $s \sqsubseteq t\ :\longleftrightarrow\ s = t{\upharpoonright} \mathsf{lh}(s)$;
        \item $s \sqsubset t :\longleftrightarrow\ s \sqsubseteq t$ and $s \not= t$;
        \item ${}^B A :=$ the set of functions from $B$ to $A$;
        \item ${}^{<\omega}A := \bigcup_{n \in \omega}{}^n A = $ the set of finite sequences in $A$.
    \end{itemize}}
\end{ob}

\begin{ob} {\rm Let $R$ be a binary relation on $X$ and $x, y \in X$. Then
\begin{itemize}
  \item $x{\uparrow}_R := \{z \in X: xRz\}$;
  \item $x{\downarrow}_R := \{z \in X: zRx\}$;
  \item $(x, y)_R := x{\uparrow}_R \cap y{\downarrow}_R$.
\end{itemize}}
\end{ob}

Now we introduce several relations on $\inftree$ and $\baireset$.

\begin{ob} {\rm Let $a, b \in \inftree \cup \baireset$. Then
    \begin{itemize}
        \item $a \triangleleft b :\longleftrightarrow \exists n \in \omega$ such that
            \begin{itemize}
                \item $a{\upharpoonright} n = b{\upharpoonright} n$ \quad and
                \item $a(n) < b(n)$.
            \end{itemize}
        \item $a \trianglelefteq b :\longleftrightarrow a \triangleleft b$ or $a = b$.
    \end{itemize}
}\end{ob}

\begin{ob} {\rm Let $\langle X, \tau \rangle$ be a topological space, $x \in X$, $B \subseteq X$, and $A \subseteq \mathbb{R}$. Then
    \begin{itemize}
        \item $\bairespace :=$ the Baire space of weight $\aleph_0 :=$ the countable power of the discrete space of cardinality $\aleph_0$;
        \item the double-arrow (two-arrow) space $:= \langle M, \tau_{\mathbb{A}} \rangle$, where $M = \{\langle x,0 \rangle:0<x\leq1\} \cup \{\langle x,1 \rangle:0\leq x < 1\}$ and $\tau_{\mathbb{A}}$ is the order topology induced by the lexicographic order on $M$ \cite[b-13 Special Spaces]{2};
        \item $\mathbb{S} :=$ the Sorgenfrey line $:= \langle \mathbb{R}, \tau_{\mathbb{S}} \rangle$, where $\tau_{\mathbb{S}}$ is the topology generated by $\{[a, b): a,b \in \mathbb{R}\}$;
        \item $A_{\mathbb{S}} :=$ the set $A$ as a subspace of $\mathbb{S}$;
        \item $A_{\mathbb{R}} :=$ the set $A$ as a subspace of $\langle \mathbb{R}, \tau_{\mathbb{R}} \rangle$, where $\tau_{\mathbb{R}}$ is the natural topology on the real line;
        \item if $y \in \mathbb{R}$, then $y < A :\longleftrightarrow y < z$ for all $z \in A$;
        \item if $p \in {}^{\omega} X$, then $p \xrightarrow{\langle X, \tau \rangle} x :\longleftrightarrow p$ converges to $x$ in $\langle X, \tau \rangle$;
        \item $\mathsf{nbhds}(x, \tau) := \{U \in \tau: x \in U\}$;
        \item $\tau {\upharpoonright} B := \{U \cap B: U \in \tau\} = $ the subspace topology of $B$;
        \item $\mathsf{Cl}_{\langle X, \tau \rangle}(B) := $ the closure of $B$ in $\langle X, \tau \rangle$;
        \item if $\langle Y, \sigma \rangle$ is a topological space, then $\langle X, \tau \rangle \cong \langle Y, \sigma \rangle :\longleftrightarrow \langle X, \tau \rangle$ is homeomorphic to $\langle Y, \sigma \rangle$.
    \end{itemize}
}\end{ob}

Recall that \cite{1} a Souslin scheme on a set $X$ is an indexed family ${\bf V} = \langle V_a \rangle_{a \in \inftree}$ of subsets of $X$.

\begin{definition}{\rm
    Let ${\bf V} = \langle V_a \rangle_{a \in \inftree}$ be a Souslin scheme on a set $X$ and $p \in \baireset$. Then

    \begin{itemize}
        \item $\mathsf{fruit}({\bf V}, p) := \bigcap_{n \in \omega}V_{p \upharpoonright n}$;
        \item ${\bf V}$ is {\it covering} $:\longleftrightarrow V_{\langle \rangle} = X$ and $V_a = \bigcup_{n \in \omega} V_{a\ \hat{}\ n}$ for all $a \in \inftree$;
        \item ${\bf V}$ is {\it complete} $:\longleftrightarrow \mathsf{fruit}({\bf V}, q) \not = \emptyset$ for all $q \in \baireset$;
        \item ${\bf V}$ has {\it strict branches} $:\longleftrightarrow |\mathsf{fruit}({\bf V}, q)| = 1$ for all $q \in \baireset$;
        \item ${\bf V}$ is {\it locally strict} $:\longleftrightarrow V_{a} = \bigcup_{n \in \omega} V_{a\ \hat{}\ n}$ and $V_{a\ \hat{}\ m} \cap V_{a\ \hat{}\ k} = \emptyset$ for all $a \in \inftree$ and $m \not = k \in \omega$.
    \end{itemize}}
\end{definition}

\begin{ob}  {\rm Let ${\bf V} = \langle V_a \rangle_{a \in \inftree}$ be a Souslin scheme on a set $X$, $x \in X$, $n \in \omega$, and $q \in \baireset$. Then
    \begin{itemize}
        \item $q$ is a {\it branch} of $x$ in ${\bf V} :\longleftrightarrow x \in \mathsf{fruit}({\bf V}, q)$;
        \item $\mathsf{branches}({\bf V}, x) := $ the set of branches of $x$ in ${\bf V}$;
        \item $\mathsf{rsubtree}(q, n) := \{s \in (q {\upharpoonright} n){\uparrow}_{\sqsubset} \cap \inftree: q \triangleleft s\}$;
        \item $\mathsf{rsequences}(q, n) := \{p \in \baireset : q \triangleleft p \land q{\upharpoonright} n = p{\upharpoonright} n\}$;
        \item $\mathsf{cut}({\bf V}, q, n) := \bigcup \{\mathsf{fruit}({\bf V}, p): p \in \mathsf{rsequences}(q, n)\}$.
    \end{itemize}}
\end{ob}

\begin{proposition} \label{Aqn}
Let ${\bf V} = \langle V_a \rangle_{a \in \inftree}$ be a covering Souslin scheme on a set $X$, let $n, m \in \omega$, $n \leq m$, and $q, p \in \baireset$. Then
    \begin{itemize}
        \item[(i)] $\mathsf{rsequences}(q, m) \subseteq \mathsf{rsequences}(q, n)$;
        \item[(ii)] if $p \trianglelefteq q$ and $p{\upharpoonright} n = q{\upharpoonright} n$, then $\mathsf{rsequences}(q, n) \subseteq \mathsf{rsequences}(p, n)$;
        \item[(iii)] $p \in \mathsf{rsequences}(q, n)$ iff $\exists k > n$ such that $p {\upharpoonright} k \in \mathsf{rsubtree}(q, n)$;
        \item[(iv)] $\mathsf{cut}({\bf V}, q, n) = \bigcup \{V_a: a \in \mathsf{rsubtree}(q, n)\} = \\
        = \{y \in V_{q \upharpoonright n}: \mathsf{branches}({\bf V}, y)\cap \mathsf{rsequences}(q, n) \not = \emptyset \}$. $\qed$
    \end{itemize}
\end{proposition}

\begin{definition}{\rm Let ${\bf V} = \langle V_a \rangle_{a \in \inftree}$ be a Souslin scheme on a set $X$, $\tau$ a topology on $X$, $x \in X$, and
$q$ is a branch of $x$ in ${\bf V}$. Then
    \begin{itemize}
        \item $q$ is a {\it $\tau$-base branch} of $x$ in ${\bf V} :\longleftrightarrow \{\mathsf{cut}({\bf V}, q, m) \cup \{x\}: m \in \omega\}$
        is an open neighborhood base at the point $x$ in the space $\langle X, \tau \rangle$.
    \end{itemize}
}
\end{definition}

Because of this definition it is natural to introduce the following notation.

\begin{ob} {\rm Let ${\bf V} = \langle V_a \rangle_{a \in \inftree}$ be a Souslin scheme on a set $X$, $x \in X$, and $q \in \baireset$. Then
$\mathsf{cutBase}({\bf V}, q, x) := \{\mathsf{cut}({\bf V}, q, m) \cup \{x\}: m \in \omega\}$.
}
\end{ob}

\begin{proposition} \label{pntcorrect}
    Let ${\bf V} = \langle V_a \rangle_{a \in \inftree}$ be a Souslin scheme on $X$ and $\tau$ a Hausdorff topology on $X$.
    Then for any sequence $q \in \baireset$ there is at most one point $x \in X$ such that $q$ is a $\tau$-base branch of $x$ in
    ${\bf V}$. $\qed$
\end{proposition}

\begin{ob} {\rm Let ${\bf V} = \langle V_a \rangle_{a \in \inftree}$ be a Souslin scheme on a set $X$, $\tau$ a topology on $X$, $x \in X$. Then
    \begin{itemize}
        \item $\mathsf{BB}({\bf V}, x, \tau)$ := the family of $\tau$-base branches of $x$ in ${\bf V}$;
        \item Let $q \in \baireset$ is a $\tau$-base branch of some point $y$ in ${\bf V}$. Then using Remark \ref{pntcorrect} we define
        $\mathsf{pnt}({\bf V}, q, \langle X, \tau \rangle) := $ the point $y \in X$ such that $q$ is a $\tau$-base branch of $y$ in ${\bf V}$.
    \end{itemize}}
\end{ob}

\begin{ob} {\rm\mbox{\ }
    \begin{itemize}
        \item ${\bf S} := $ the Souslin scheme $\langle S_a \rangle_{a \in \inftree}$ on $\baireset$ such that $S_a := \{p \in \baireset: a \sqsubseteq p\}$ for all $a \in \inftree$.
    \end{itemize}}
\end{ob}

\begin{proposition} \label{p_eq_branchp}
    $\mathsf{branches}({\bf S}, p) = \{p\}$ for all $p \in \baireset$. $\qed$
\end{proposition}

Let $\langle X, \tau \rangle$ be a topological space. Then a Souslin scheme ${\bf V} = \langle V_a \rangle_{a \in \inftree}$ on the set $X$ is called an
{\it open} Souslin scheme on the space $\langle X, \tau \rangle$ if $V_a \in \tau$ for all $a \in \inftree$.

\begin{definition} {\rm
Let $\langle X, \tau \rangle$ be a Hausdorff topological space. Then an open complete covering Souslin scheme ${\bf V} = \langle V_a \rangle_{a \in \inftree}$ on the space $\langle X, \tau \rangle$ is called a {\it Sorgenfrey base} for $\langle X, \tau \rangle$ if the following conditions hold:
\begin{itemize}
  \item[(S1)] $\forall x \in X\ \forall q \in \mathsf{branches}({\bf V}, x)\ \forall n \in \omega\\ \exists t \in \mathsf{BB}({\bf V}, x, \tau)\
  (t {\upharpoonright} n = q {\upharpoonright} n )$;
  \item[(S2)] $\forall q \in \baireset\ \exists z \in X (q \in \mathsf{BB}({\bf V}, z, \tau))$.
\end{itemize}
}
\end{definition}

\begin{proposition}
    Let ${\bf V} = \langle V_a \rangle_{a \in \inftree}$ be a Sorgenfrey base for a Hausdorff topological space $\langle X, \tau \rangle$. Then $\mathsf{branches}({\bf V}, x) \subseteq \mathsf{Cl}_{\bairespace}(\mathsf{BB}({\bf V}, x, \tau))$ for all $x \in X$. $\qed$
\end{proposition}

\section{Description of open images of the Sorgenfrey line}
The main result of this section is the following theorem.

\begin{theorem} \label{desc_op_im}
A Hausdorff space is a continuous open image of the Sorgenfrey line iff there exists a Sorgenfrey base for this space.
\end{theorem}

\begin{proof}
    The theorem follows from Corollary \ref{st_vid_S}, Lemma \ref{image_has_SB}, and Lemma \ref{space_with_SB_image}.
\end{proof}

\begin{lemma} \label{predststr}
    The family $\bigcup \{\mathsf{cutBase}({\bf S}, p, p) : p \in \baireset\}$ is a base for a topology on $\baireset$ and for any point $x \in \baireset$ the family
    $\mathsf{cutBase}({\bf S}, x, x)$ is a neighborhood base at the point x in this topology.
\end{lemma}

\begin{proof}
    Let $p, q \in \baireset$, let $n, m \in \omega$, and let $x \in (\mathsf{cut}({\bf S}, p, n)\cup\{p\}) \cap (\mathsf{cut}({\bf S}, q, m)\cup\{q\})$. We now prove that
    there exists a set $U \in \mathsf{cutBase}({\bf S}, x, x)$ such that

    \begin{equation}
        x \in U \subseteq (\mathsf{cut}({\bf S}, p, n)\cup\{p\}) \cap (\mathsf{cut}({\bf S}, q, m)\cup\{q\}).
    \end{equation}

    Without loss of generality, we can assume that $n \leq m$. Consider the set $U := \mathsf{cut}({\bf S}, x, m) \cup \{x\}$.
    Let $z \in U \setminus \{x\}$. Then using (i), (ii) of Remark \ref{Aqn} and Remark \ref{p_eq_branchp}, we get $z \in \mathsf{rsequences}(x, m) \subseteq
    \mathsf{rsequences}(p, n) \cap \mathsf{rsequences}(q, m)$. Finally, from (i), (ii) of Remark \ref{Aqn} and Remark \ref{p_eq_branchp} it follows that $z \in (\mathsf{cut}({\bf S}, p, n)
    \cup\{p\}) \cap (\mathsf{cut}({\bf S}, q, m)\cup\{q\})$.
\end{proof}

\begin{ob}
    $\sigma_{\mathbb{S}} :=$ the topology on $\baireset$ that is constructed in Lemma \ref{predststr}.
\end{ob}

\begin{proposition}
    ${\bf S}$ is an open complete covering Souslin scheme on the space $\langle \baireset, \sigma_{\mathbb{S}} \rangle$. $\qed$
\end{proposition}

\begin{lemma} \label{image_has_SB}
    Let $\langle Y, \tau \rangle$ be a Hausdorff space and let $f: \langle \baireset, \sigma_{\mathbb{S}} \rangle \rightarrow \langle Y, \tau \rangle$ be an
    open continuous surjection. Denote $f[S_a]$ by $V_a$ for all $a \in \inftree$.
    Then ${\bf V} := \langle V_a \rangle_{a \in \inftree}$ is a Sorgenfrey base for $\langle Y, \tau \rangle$.
\end{lemma}

\begin{proof}
    Since $f$ is an open surjection and ${\bf S}$ is an open complete covering Souslin scheme, we see that ${\bf V}$ is an open complete covering Souslin scheme on the space
    $\langle Y, \tau \rangle$.

    Let us prove that

    \begin{equation}\label{seq_in_seqs}
        x \in \mathsf{branches}({\bf V}, y) \qquad \mbox{ for all $y \in Y$ and $x \in f^{-1}(y)$}.
    \end{equation}

    Let $y \in Y$ and $x \in f^{-1}(y)$. Then for any $n \in \omega$ it follows that $y = f(x) \in f[S_{x \upharpoonright n}] = V_{x \upharpoonright n}$,
    and hence $x \in \mathsf{branches}({\bf V}, y)$.

    We now prove that

    \begin{equation}\label{q_x_in_SBB}
         x \in \mathsf{BB}({\bf V}, y, \tau)  \qquad \mbox{ for all $y \in Y$ and $x \in f^{-1}(y)$}.
    \end{equation}

    Let $y \in Y$ and $x \in f^{-1}(y)$. Since $\mathsf{cutBase}({\bf S}, x, x)$ is an open neighborhood base at the poin $x$ and $f$ is an open map,
    it is enough to prove that $\mathsf{cut}({\bf V}, x, n) \cup \{y\} = f[\mathsf{cut}({\bf S}, x, n) \cup \{x\}]$ for all $n \in \omega$.

    Fix $n \in \omega$. Using (iv) of Remark \ref{Aqn}, we get
    \[
    \begin{split}
        &f[\mathsf{cut}({\bf S}, x, n) \cup \{x\}] =\\
        &f[{\textstyle \bigcup} \{S_a: a \in \mathsf{rsubtree}(x, n)\}] \cup \{y\} =\\
        &{\textstyle \bigcup} \{f[S_a]: a \in \mathsf{rsubtree}(x, n)\} \cup \{y\} =\\
        &{\textstyle \bigcup} \{V_a : a \in \mathsf{rsubtree}(x, n)\} \cup \{y\} =\\
        &\mathsf{cut}({\bf V}, x, n) \cup \{y\}.
    \end{split}
    \]
    Since $f$ is a surjection, it follows from (\ref{q_x_in_SBB}) that ${\bf V}$ satisfies (S2).

    We now prove that ${\bf V}$ satisfies (S1). Let $y \in Y$, $q \in \mathsf{branches}({\bf V}, y)$, and $n \in \omega$. Then $y \in V_{q \upharpoonright n} = f[S_{q \upharpoonright n}]$. This means that there exists $x \in S_{q \upharpoonright n}$ such that $f(x) = y$, so it follows from (\ref{q_x_in_SBB}) that $x \in \mathsf{BB}({\bf V}, y, \tau)$. Also since $x \in S_{q \upharpoonright n}$, we see that $x {\upharpoonright} n = q {\upharpoonright} n$.
\end{proof}

\begin{lemma} \label{space_with_SB_image}
    Let $\langle Y, \tau \rangle$ be a Hausdorff space with a Sorgenfrey base. Then there exists a continuous open surjection $f: \langle \baireset, \sigma_{\mathbb{S}} \rangle \rightarrow \langle Y, \tau \rangle$. Moreover, if $\langle Y, \tau \rangle$ has a locally strict Sorgenfrey base with strict branches, then $\langle Y, \tau \rangle$ is homeomorphic to $\langle \baireset, \sigma_{\mathbb{S}} \rangle$.
\end{lemma}

\begin{proof}
    Let ${\bf V} = \langle V_a \rangle_{a \in \inftree}$ be a Sorgenfrey base for $\langle Y, \tau \rangle$.

    Let $f$ be the map from $\baireset$ to $Y$ such that $f(p) := \mathsf{pnt}({\bf V}, p, \langle Y, \tau \rangle)$ for all $p \in \baireset$. From property (S2) of ${\bf V}$ it follows that $f$ is a surjection. Let us prove that $f: \langle \baireset, \sigma_{\mathbb{S}} \rangle \rightarrow \langle Y, \tau \rangle$ is continuous and open.

    Let $p \in \baireset$. Since $p \in \mathsf{BB}({\bf V}, f(p), \tau)$, we see that $\mathsf{cutBase}({\bf V}, p, f(p))$ is an open neighborhood base at the point $f(p)$ in $\langle Y, \tau \rangle$. Hence we must only prove that

    \begin{equation}
        f[\mathsf{cut}({\bf S}, p, n) \cup \{p\}] = \mathsf{cut}({\bf V}, p, n) \cup \{f(p)\} \mbox{ for all $n \in \omega$}.
    \end{equation}

    Let $n \in \omega$. We will prove two inclusions.

    \begin{itemize}
        \item[$"\subseteq"$] Let $q \in \mathsf{cut}({\bf S}, p, n)$. From the definition of $f$ it follows that $q \in \mathsf{branches}({\bf V}, f(q))$. From Remark \ref{p_eq_branchp} and (iv) of Remark \ref{Aqn} it follows that $q {\upharpoonright} n = p {\upharpoonright} n$ and $p \triangleleft q$. Finally, $f(q) \in \mathsf{cut}({\bf V}, p, n)$.
        \item[$"\supseteq"$] Let $x \in \mathsf{cut}({\bf V}, p, n)$. From (iv) of Remark \ref{Aqn} it follows that there exists $q \in \mathsf{branches}({\bf V}, x)$ such that $q {\upharpoonright} n = p {\upharpoonright} n$ and $p \triangleleft q$. Choose $m \in \omega$ such that $q {\upharpoonright} m = p {\upharpoonright} m$ and $p {\upharpoonright} (m+1) \triangleleft q {\upharpoonright} (m+1)$. From property (S1) of ${\bf V}$ it follows that there exists $t \in \mathsf{BB}({\bf V}, x, \tau)$ such that $t {\upharpoonright} (m+1) = q {\upharpoonright} (m+1)$. Hence $t {\upharpoonright} n = q {\upharpoonright} n = p {\upharpoonright} n$ and $p \triangleleft t$, so $t \in \mathsf{cut}({\bf S}, p, n)$ by (iv) of Remark \ref{Aqn}. Since $t \in \mathsf{BB}({\bf V}, x, \tau)$, we see that $f(t) = x$.
    \end{itemize}

    Now suppose that ${\bf V}$ is locally strict and has strict branches. Then for any $x \in Y$ the set
    $\mathsf{branches}({\bf V}, x)$ is a singleton, therefore $f$ is a bijection because
    $f^{-1}(x) \subseteq \mathsf{BB}({\bf V}, x, \tau) \subseteq \mathsf{branches}({\bf V}, x)$.
\end{proof}

Now we show that the Sorgenfrey line has a locally strict Sorgenfrey base with strict branches. To be precise, we show that a Lusin $\pi$-base for the Sorgenfrey line that was constructed in \cite{3} has all this properties.

\begin{prim}{\rm
We build a Souslin scheme ${\bf V}^{\mathbb{S}} = \langle V^{\mathbb{S}}_a \rangle_{a \in \inftree}$ by recursion on $\mathsf{lh}(a)$.

Let ${\bf V}^{\mathbb{S}}_{\langle \rangle} := \mathbb{R}$, and let the set $\{{\bf V}^{\mathbb{S}}_a: \mathsf{lh}(a)=1\}$ equals $\{[i, i+1): i \in \mathbb{Z}\}$. When $\mathsf{lh}(a) \geq 1$, consider an interval $V^{\mathbb{S}}_a = [i, j)$. Let
$\langle x_n \rangle_{n \in \omega}$ be a sequence in $\mathbb{R}$ such that $\langle x_n \rangle_{n \in \omega} \xrightarrow{\langle \mathbb{R}, \tau_{\mathbb{R}} \rangle} j$, $x_0 = i$, $x_{n+1} >
x_n$, and $x_{n+1} - x_n \leq \frac{1}{\mathsf{lh}(a) + 1}$. Define $V^{\mathbb{S}}_{a\ \hat{}\ n} := [x_n, x_{n+1})$.
}\end{prim}

The next Corollary was first observed by Mikhail Patrakeev in private correspondence.

\begin{corollary} \label{st_vid_S}
$\mathbb{S} \cong \langle \baireset, \sigma_{\mathbb{S}} \rangle$.
\end{corollary}

\begin{proof}
The reader will easily prove that ${\bf V}^{\mathbb{S}}$ is a locally strict Sorgenfrey base for $\mathbb{S}$ that has strict branches. Then from Lemma
\ref{space_with_SB_image} it follows that $\mathbb{S} \cong \langle \baireset, \sigma_{\mathbb{S}} \rangle$.
\end{proof}

\begin{corollary} \label{descS}
    The Sorgenfrey line is, up to homeomorphism, the unique Hausdorff topological space that has a locally strict Sorgenfrey base with strict branches. $\qed$
\end{corollary}

\section{The spaces that are R-bidirected along Q}

\begin{definition}{\rm
    Let $\langle X, \tau \rangle$ be a topological space, $Q$ a dense subset of $\langle X, \tau \rangle$, $R$ a binary relation on $X$, and
    $x \in X$. Then:
    \begin{itemize}
        \item An open neighborhood $U$ of $x$ is {\it $R$-right along $Q$} if $x R y$ for all $y \in (U \setminus \{x\}) \cap Q$;
        \item we say that $x$ {\it looks to the $R$-right along $Q$} if the following conditions hold:
            \begin{itemize}
                \item there exists an $R$-right along $Q$ open neighborhood of $x$;
                \item for any open neighborhood $U$ of $x$ there exists $y \in (U \setminus \{x\}) \cap Q$ such that $y {\downarrow}_R$ is a neighborhood of $x$.
            \end{itemize}
        \item An open neighborhood $U$ of $x$ is {\it $R$-left along $Q$} if $y R x$ for all $y \in (U \setminus \{x\}) \cap Q$;
        \item we say that $x$ {\it looks to the $R$-left along $Q$} if the following conditions hold:
            \begin{itemize}
                \item there exists an $R$-left along $Q$ open neighborhood of $x$;
                \item for any open neighborhood $U$ of $x$ there exists $y \in (U \setminus \{x\}) \cap Q$ such that $y {\uparrow}_R$ is a neighborhood of $x$.
            \end{itemize}
    \end{itemize}
}\end{definition}

Recall that a binary relation $R$ on a set $X$ is {\it asymmetric} if $xRy \Rightarrow \neg yRx$ for all $x, y \in X$.

\begin{definition}\label{bidir}{\rm
    Let $\langle X, \tau \rangle$ be a topological space, $Q$ a dense subset of $\langle X, \tau \rangle$, and $R$ an asymmetric binary relation on $X$. Then
    we say that $\langle X, \tau \rangle$ is {\it R-bidirected along Q} if there are dense subsets $A_l, A_r$ of $\langle X, \tau \rangle$ such that
    \begin{itemize}
        \item $X = A_l \cup A_r$ and $A_l \cap A_r = \emptyset$;
        \item $x$ looks to the $R$-right along Q for all $x \in A_r$;
        \item $x$ looks to the $R$-left along Q for all $x \in A_l$.
    \end{itemize}
}\end{definition}

\begin{prim}{\rm
It is easy to check that the double-arrow space is $R$-bidirected along itself, where $R$ is the strict lexicographic order on it.}
\end{prim}

\begin{lemma} \label{main_tech_lem}
    Let $\langle X, \tau \rangle$ be a Hausdorff space. Suppose that there are a dense subset $Q$ of $\langle X, \tau \rangle$ and an asymmetric binary relation $R$ on $X$ such that $\langle X, \tau \rangle$ is R-bidirected along Q. Then $\langle X, \tau \rangle$ is not a continuous open image of $\mathbb{S}$.
\end{lemma}

\begin{proof}
    Assume the converse. From Theorem \ref{desc_op_im} it follows that there exists ${\bf V} = \langle V_a \rangle_{a \in \inftree}$ is a Sorgenfrey base for $\langle X, \tau \rangle$. Let $A_l$ and $A_r$ be the sets from Definition \ref{bidir}. We build, by recursion on $n$, a sequence $\langle x_n \rangle_{n \in \omega}$ in $Q$, a sequence $\langle t_n \rangle_{n \in \omega}$ in $\baireset$, and an $\sqsubset$-increasing sequence $\langle p_n \rangle_{n \in \omega}$ in $\inftree$ such that

    \begin{equation}\label{property_of_p}
         \forall k \in \omega \exists x_1, x_2 \in \mathsf{cut}({\bf V}, p, k) \cap Q \mbox{ such that } x_1 R z R x_2 \mbox{ for all } z \in \mathsf{fruit}({\bf V}, p),
    \end{equation}
    where $p = \bigcup_{n \in \omega} p_n$.

    Take $p_0 := \langle \rangle$. Suppose we have constructed $p_0, \dots, p_n$; $x_0, \dots, x_{n-1}$; and $t_{0}, \dots, t_{n-1}$. Now we consider two cases. Let $n$ be even. Then take any point $x \in V_{p_n} \cap A_r$. From property (S1) of ${\bf V}$ it follows that there exists $q \in \mathsf{BB}({\bf V}, x, \tau)$ such that

    \begin{equation}\label{q__lenpn_eq_p_n}
        p_n \sqsubset q.
    \end{equation}

    Take $m > \mathsf{lh}(p_n)$ such that $\mathsf{cut}({\bf V}, q, m) \cup \{x\} \subseteq V_{p_n}$. Since $x$ looks to the $R$-right along $Q$, we can take $x_n \in (\mathsf{cut}({\bf V}, q, m) \setminus \{x\}) \cap Q$ such that

    \begin{equation}\label{x_n_greater_U}
         x_n{\downarrow}_R \mbox{ is a neighborhood of } x.
    \end{equation}

    From (iv) of Remark \ref{Aqn} it follows that we can take $t_n \in \mathsf{branches}({\bf V}, x_n)$ such that

    \begin{equation}\label{property_of_t}
         t_n \in \mathsf{rsequences}(q, m).
    \end{equation}

    Consider $k > m$ such that

    \begin{equation}\label{tk_greater_qk}
        q {\upharpoonright} k \triangleleft t_n {\upharpoonright} k.
    \end{equation}

    Since $q \in \mathsf{BB}({\bf V}, x, \tau)$, we see that there exists $a > k$ such that

    \begin{equation}\label{cut_in_U}
        \mathsf{cut}({\bf V}, q, a) \cup \{x\} \subseteq x_n{\downarrow}_R.
    \end{equation}

    Let $p_{n + 1}$ be any element of $\mathsf{rsubtree}(q, a)$. From $(\ref{q__lenpn_eq_p_n})$ and inequalities $a > k > m > \mathsf{lh}(p_n)$ it follows that $p_n = q {\upharpoonright} \mathsf{lh}(p_n) \sqsubset q {\upharpoonright} a \sqsubset p_{n+1}$. Since $p_{n + 1} \in \mathsf{rsubtree}(q, a)$, from (iv) of Remark \ref{Aqn} and (\ref{cut_in_U}) it follows that

    \begin{equation}\label{prop_of_xn_chet}
        z R x_n \mbox{ for all } z \in V_{p_{n+1}}.
    \end{equation}

    Now we show that

    \begin{equation}\label{prop_of_tn_chet}
        p_n \sqsubset t_n \mbox{ and } p_{n+1} \triangleleft t_n {\upharpoonright} \mathsf{lh}(p_{n+1}).
    \end{equation}

    From inequality $m > \mathsf{lh}(p_n)$, (\ref{q__lenpn_eq_p_n}), and (\ref{property_of_t}) it follows that $p_n \sqsubset t_n$. To prove the second part of (\ref{prop_of_tn_chet}), we must observe that $q {\upharpoonright} k \sqsubset p_{n+1}$ and use (\ref{tk_greater_qk}).

    Now let $n$ be odd. If we argue as above by taking $x \in V_{p_n} \cap A_l$, $x_n$ such that $x_n{\uparrow}_R$ is a neighborhood of $x$, and $t_n \in \mathsf{branches}({\bf V}, \mathsf{x}(n))$, then we can choose $p_{n+1}$ such that

    \begin{equation}\label{prop_of_xn_nechet}
        x_n R z \mbox{ fol all } z \in V_{p_{n+1}};
    \end{equation}

    and

    \begin{equation}\label{prop_of_tn_nechet}
        p_n \sqsubset t_n \mbox{ and } p_{n+1} \triangleleft t_n {\upharpoonright} \mathsf{lh}(p_{n+1}).
    \end{equation}

    We now prove (\ref{property_of_p}). Let $k \in \omega$, $n > k$, and $n$ is even. Then $p {\upharpoonright} k \sqsubset p_n$. So from (\ref{prop_of_tn_chet}) it follows that $t_n \in \mathsf{rsequences}(p, k)$. Hence using (iv) of Remark \ref{Aqn}, we get $x_n \in \mathsf{cut}({\bf V}, p, k)$. Finally, from (\ref{prop_of_xn_chet}) it follows that $z R x_n$ for all $z \in \mathsf{fruit}({\bf V}, p) \subseteq V_{p_{n+1}}$. If we argue as above by taking $m > k$ such that $m$ is odd, then we can show that $x_m \in \mathsf{cut}({\bf V}, p, k)$ and $x_m R z$ for all $z \in \mathsf{fruit}({\bf V}, p) \subseteq V_{p_{m+1}}$.

    Let us prove that

    \begin{equation}\label{contr}
        \forall z \in \mathsf{fruit}({\bf V}, p) : p \not \in \mathsf{BB}({\bf V}, z, \tau).
    \end{equation}

    Let $z \in \mathsf{fruit}({\bf V}, p)$, we consider two cases.

    Case 1: $z \in A_r$. From (\ref{property_of_p}) it follows that for any $k \in \omega$ there exists $x_1 \in \mathsf{cut}({\bf V}, p, k) \cap Q$ such that $x_1 R z$, and so $\neg zRx_1$ and $z \not= x_1$. Hence for all $k \in \omega$ if $\mathsf{cut}({\bf V}, p, k) \cup \{z\}$ is an open neighborhood of $z$, then it is not $R$-right along $Q$. And since $z$ looks to the $R$-right along $Q$, we see that $\mathsf{cutBase}({\bf V}, p, z)$ is not an open neighborhood base at the point $z$. So $p \not \in \mathsf{BB}({\bf V}, z, \tau)$.

    Case 2: $z \in A_l$. Arguing as above, we can take $x_2 \in \mathsf{cut}({\bf V}, p, k) \cap Q$ such that $z R x_2$, and so $p \not \in \mathsf{BB}({\bf V}, z, \tau)$.

    Formula (\ref{contr}) contradicts property (S2) of ${\bf V}$. The lemma is proved.
\end{proof}

\section{The spaces that are not open images of the Sorgenfrey line}

A subset of a topological space is called {\it co-dense} if its complement is dense.

\begin{theorem} \label{densecodense}
    Suppose that $\langle X, \tau \rangle$ is a Hausdorff topological space, $S \subseteq X$ is a dense and co-dense subset of $X$. If $\langle S, \tau {\upharpoonright} S \rangle$ is homeomorphic to the Sorgenfrey line, then $\langle X, \tau \rangle$ is not a continuous open image of the Sorgenfrey line.
\end{theorem}

\begin{proof}
    Since $[0,1)_{\mathbb{S}} \cong \mathbb{S}$, without loss of generality, we can assume that $S = [0,1)$ and $[0,1)_{\mathbb{S}} = \langle S, \tau {\upharpoonright} S \rangle$.

    If $\langle X, \tau \rangle$ is not first-countable, then it is not a continuous open image of the Sorgenfrey line, because the first axiom of countability is preserved by continuous open maps. So suppose that $\langle X, \tau \rangle$ is first-countable.

    Using Lemma \ref{main_tech_lem} it is enough to prove that there exists an asymmetric relation $R$ on $X$ such that $\langle X, \tau \rangle$ is $R$-bidirected along $[0, 1)$. For each $A \subseteq X$, by $\breve{A}$ we denote $A \cap [0,1)$.

    Denote by $\mathsf{L}$ the function with domain $X \setminus [0, 1)$ such that for all $z \in X \setminus [0, 1)$
        $$
            \mathsf{L}(z) := \{x \in [0, 1] : \exists p \in {}^{\omega}[0, 1)(p \xrightarrow{\langle X, \tau \rangle} z\ \land\ p \xrightarrow{[0,1]_{\mathbb{R}}} x)\}.
        $$

    Since $\langle X, \tau \rangle$ is a first countable space and $[0,1)$ is dense in it, we see that $\mathsf{L}(z) \not = \emptyset$ for all $z \in X \setminus [0, 1)$. We prove that

    \begin{equation}\label{le_ord_lz}
      < \mbox{ well-orders } \mathsf{L}(z) \text{ for all } z \in X \setminus [0, 1).
    \end{equation}

    Let $z \in X \setminus [0, 1)$. Assume the converse. Let $q \in {}^{\omega}\mathsf{L}(z)$ be such that $q(n + 1) < q(n)$ for all $n \in \omega$. Consider $x \in [0, 1)$ such that

    \begin{equation}\label{xn_conv_x}
        q \xrightarrow{[0,1)_{\mathbb{S}}} x.
    \end{equation}

    Let us prove that

    \begin{equation}\label{xz}
        U_z \cap U_x \not = \emptyset \mbox{ for all } U_x \in \mathsf{nbhds}(x, \tau) \mbox{ and } U_z \in \mathsf{nbhds}(z, \tau).
    \end{equation}

    Let $U_x \in \mathsf{nbhds}(x, \tau)$ and $U_z \in \mathsf{nbhds}(z, \tau)$. Without loss of generality, we can assume that $\breve{U}_x = [x, x + \varepsilon)$. From (\ref{xn_conv_x}) it follows that there exists $n \in \omega$ such that $q(n) \in \breve{U}_x \subseteq U_x$. Since $q(n) \in \mathsf{L}(z)$, there exists $p \in {}^{\omega}[0, 1)$ such that $p \xrightarrow{\langle X, \tau \rangle} z$ and  $p \xrightarrow{[0,1]_{\mathbb{R}}} q(n)$. Since $q(n) \not = x$, we obtain $q(n) \in (x, x + \varepsilon)$, and so $p$ is eventually in $(x, x + \varepsilon) \subseteq U_x$. Also $p$ is eventually in $U_z$, hence $U_x \cap U_z \not = \emptyset$. Formula (\ref{xz}) contradicts the Hausdorff property of $\langle X, \tau \rangle$, so (\ref{le_ord_lz}) is proved.

    Denote by $\mathsf{m}$ the function that takes each point $z \in X \setminus [0,1)$ to the $<$-minimal element of $\mathsf{L}(z)$. Also let $\mathsf{M}$ be the function that takes each point $z \in X \setminus [0,1)$ to the $\mathsf{sup}_{[0,1]}(\mathsf{L}(z))$. Now we prove a technical lemma about this functions.

    \begin{lemma} \label{lemma_in_th} Let $z \in X \setminus [0, 1)$ and $a, b \in \mathbb{R}$ such that $a<b$. Then
        \begin{itemize}
          \item[(i)] $\exists U \in \mathsf{nbhds}(z, \tau)(\breve{U} < \mathsf{M}(z))$;
          \item[(ii)] $\forall U \in \mathsf{nbhds}(z, \tau)(\breve{U} = [a,b) \rightarrow \mathsf{M}(z) \leq b)$;
          \item[(iii)] $\forall x \in [0,1)(x < \mathsf{m}(z) \rightarrow \exists U \in \mathsf{nbhds}(z, \tau)(x < \breve{U}))$;
          \item[(iv)]$\forall x \in [0,1)\ \forall U \in \mathsf{nbhds}(z, \tau) (x < \breve{U} \rightarrow x < \mathsf{m}(z))$;
          \item[(v)] $\forall U \in \mathsf{nbhds}(z, \tau)\exists x \in \breve{U}(x < \mathsf{m}(z))$.
        \end{itemize}
    \end{lemma}

    \begin{proof}
        (i). Assume the converse. Assume that for any $U \in \mathsf{nbhds}(z, \tau)$ there exists $x \in \breve{U}$ such that $\mathsf{M}(z) \leq x$. Let $p \in {}^{\omega}[0,1)$ be such that

        \begin{equation} \label{a}
            \forall n \in \omega (\mathsf{M}(z) \leq p(n))
        \end{equation}
        and
        \begin{equation} \label{b}
            p \xrightarrow{\langle X, \tau \rangle} z.
        \end{equation}

        Take $x \in [0,1]$ such that $p\sp{\prime} \xrightarrow{[0,1]_{\mathbb{R}}} x$, where $p\sp{\prime}$ is a subsequence of $p$. Then from (\ref{b}) it follows that $p\sp{\prime} \xrightarrow{\langle X, \tau \rangle} z$, and so $x \in \mathsf{L}(z)$. From (\ref{a}) it follows that $\mathsf{M}(z) \leq x$. From equality $\mathsf{M}(z) = \mathsf{sup}_{[0,1]}(\mathsf{L}(z))$ it follows that $x = \mathsf{M}(z)$, and so using (\ref{a}), we get $p\sp{\prime} \xrightarrow{[0,1)_{\mathbb{S}}} \mathsf{M}(z)$. Hence $p\sp{\prime} \xrightarrow{\langle X, \tau \rangle} \mathsf{M}(z)$. This contradicts the Hausdorff property of $\langle X, \tau \rangle$, so we prove (i).

        (ii). Assume the converse. Suppose that there exists $U \in \mathsf{nbhds}(z, \tau)$ such that $\breve{U} = [a,b)$ and $b < \mathsf{M}(z)$. Since $\mathsf{M}(z) = \mathsf{sup}_{[0,1]}(\mathsf{L}(z))$, we see that there exists $q \in {}^{\omega} \mathsf{L}(z)$ such that

        $$
            q \xrightarrow{[0,1]_{\mathbb{R}}} \mathsf{M}(z)
        $$
        and
        $$
            q(n) \leq \mathsf{M}(z) \text{ for all } n \in \omega.
        $$

        Hence there exists $m \in \omega$ such that $b < q(m) \leq \mathsf{M}(z)$. $q(m) \in \mathsf{L}(z)$, so there exists $p \in {}^{\omega}[0,1)$ such that

        \begin{equation} \label{c}
            p \xrightarrow{\langle X, \tau \rangle} z
        \end{equation}
        and
        \begin{equation} \label{d}
            p \xrightarrow{[0,1]_{\mathbb{R}}} q(m).
        \end{equation}

        Let $\varepsilon > 0$ be such that $b < (q(m) - \varepsilon, q(m) + \varepsilon)$. Then $(q(m) - \varepsilon, q(m) + \varepsilon) \cap U = \emptyset$ and $p$ is eventually in $(x_m - \varepsilon, x_m + \varepsilon)$. This contradicts formula (\ref{c}), so we prove (ii).

        (iii). Assume the converse. Suppose that there exists $x \in [0,1)$ such that $x < \mathsf{m}(z)$ and $\forall U \in \mathsf{nbhds}(z, \tau)$ there exists $y \in U$ such that $y \leq x$. Take $p \in {}^{\omega}[0,1)$ such that

        \begin{equation} \label{e}
            p \xrightarrow{\langle X, \tau \rangle} z
        \end{equation}
        and
        \begin{equation} \label{f}
            p(n) \leq x \text{ for all } n \in \omega.
        \end{equation}

        Consider $y \in [0,1]$ such that $p\sp{\prime} \xrightarrow{[0,1]_{\mathbb{R}}} y$, where $p\sp{\prime}$ is a subsequence of $p$. Then from (\ref{e}) it follows that $p\sp{\prime} \xrightarrow{\langle X, \tau \rangle} z$, and so $y \in \mathsf{L}(z)$. From (\ref{f}) it follows that $y \leq x < \mathsf{m}(z)$, this contradicts $\mathsf{m}(z)$ is the $<$-minimal element of $\mathsf{L}(z)$, so we prove (iii).

        (iv). Assume the converse. Suppose that there exists $x \in [0,1)$ and $U \in \mathsf{nbhds}(z, \tau)$ such that $x < \breve{U}$ and $\mathsf{m}(z) \leq x$. Since $\mathsf{m}(z) \in \mathsf{L}(z)$, we see that there exists $p \in {}^{\omega}[0,1)$ such that
        \begin{equation} \label{g}
            p \xrightarrow{\langle X, \tau \rangle} z
        \end{equation}
        and
        \begin{equation} \label{h}
            p \xrightarrow{[0,1]_{\mathbb{R}}} \mathsf{m}(z).
        \end{equation}

        Now let us consider two cases. Case 1: $x = \mathsf{m}(z)$. Since $x < \breve{U}$, from (\ref{g}) it follows that there exists $n \in \omega$ such that $x < p(k)$ for all $k > n$. Hence using (\ref{h}) and equality $x = \mathsf{m}(z)$, we get $p \xrightarrow{[0,1)_{\mathbb{S}}} x$, and so $p \xrightarrow{\langle X, \tau \rangle} x$. This contradicts the Hausdorff property of $\langle X, \tau \rangle$.

        Case 2: $x > \mathsf{m}(z)$. Take $\varepsilon > 0$ such that $(\mathsf{m}(z) - \varepsilon, \mathsf{m}(z) + \varepsilon) < x$. Since $x < \breve{U}$, we see that $(\mathsf{m}(z) - \varepsilon, \mathsf{m}(z) + \varepsilon) \cap U = \emptyset$. From (\ref{h}) it follows that $p$ is eventually in $(\mathsf{m}(z) - \varepsilon, \mathsf{m}(z) + \varepsilon)$, but this contradicts formula (\ref{g}), so we prove (iv).

        (v). Assume the converse. Suppose that there exists $U \in \mathsf{nbhds}(z, \tau)$ such that
        \begin{equation} \label{i}
            \mathsf{m}(z) \leq x \text{ for all } x \in \breve{U}.
        \end{equation}

        From $\mathsf{m}(z) \in \mathsf{L}(z)$ it follows that there exists $p \in {}^{\omega}[0,1)$ such that
        \begin{equation} \label{j}
            p \xrightarrow{\langle X, \tau \rangle} z
        \end{equation}
        and
        \begin{equation} \label{k}
          p \xrightarrow{[0,1]_{\mathbb{R}}} \mathsf{m}(z).
        \end{equation}

        From (\ref{j}) it follows that $p$ is eventually in $U$, and so from (\ref{i}) and (\ref{k}) it follows that $p \xrightarrow{[0,1)_{\mathbb{S}}} \mathsf{m}(z)$, hence $p \xrightarrow{\langle X, \tau \rangle} \mathsf{m}(z)$. This contradicts the Hausdorff property of $\langle X, \tau \rangle$, so we prove (v).
    \end{proof}
    Define

    \[
    \begin{split}
        xRy \iff &x < y\ \lor \\
                 &y \in X \setminus [0,1) \land x < \mathsf{M}(y)\ \lor \\
                 &x \in X \setminus [0,1) \land \mathsf{M}(x) \leq y.
    \end{split}
    \]

    It is not hard to prove that $R$ is asymmetric. Define $A_r := [0,1)$ and $A_l := X \setminus [0,1)$. Note that

    \begin{equation}\label{xaral}
        X = A_r \cup A_l \text{ and } A_r \cap A_l = \emptyset.
    \end{equation}

    Let us prove that

    \begin{equation} \label{xright}
        x \text{ looks to the } R \text{-right along } [0,1) \text{ for all } x \in A_r.
    \end{equation}

    Let $x \in A_r = [0,1)$. There exists $W \in \mathsf{nbhds}(x, \tau)$ such that $\breve{W} = [x,1)$. It is easy to prove that $W$ is $R$-right along $[0,1)$.

    Now let $U$ be an arbitrary element of $\mathsf{nbhds}(x, \tau)$. Take $y \in \breve{U}$ such that $x < y$ and $V \in \mathsf{nbhds}(x, \tau)$ such that $\breve{V} = [x, y)$. Let us prove that
    \begin{equation*}
        V \subseteq y {\downarrow}_R.
    \end{equation*}

    Let $z \in V$. If $z \in \breve{V}$, then $z < y$, and so $zRy$. Suppose that $z \in V \setminus [0,1)$. Then from (ii) of Lemma \ref{lemma_in_th} it follows that $\mathsf{M}(z) \leq y$, and so by definition of $R$, we get $zRy$.

    Now we shall prove that

    \begin{equation}\label{zleft}
        z \text{ looks to the } R \text{-left along } [0,1) \text{ for all } z \in A_l.
    \end{equation}

    Let $z \in A_l = X \setminus [0,1)$. From (i) of Lemma \ref{lemma_in_th} it follows that there exists $U \in \mathsf{nbhds}(z, \tau)$ such that $\breve{U} < \mathsf{M}(z)$, and so by definition of $R$, we get $xRz$ for all $x \in \breve{U}$, hence $U$ is an $R$-left along $[0,1)$.

    Now let $U$ be an arbitrary element of $\mathsf{nbhds}(z, \tau)$. From (v) of Lemma \ref{lemma_in_th} it follows that there exists $x \in \breve{U}$ such that

    \begin{equation}\label{xlemz}
        x < \mathsf{m}(z).
    \end{equation}

    From (\ref{xlemz}) and (iii) of Lemma \ref{lemma_in_th} it follows that there exists $V \in \mathsf{nbhds}(z, \tau)$ such that

    \begin{equation}\label{xleVz}
        x < \breve{V}.
    \end{equation}

    Let us prove that

    \begin{equation*}
        V \subseteq x {\uparrow}_R.
    \end{equation*}

    Let $y \in V$. If $y \in \breve{V}$, then from (\ref{xleVz}) it follows that $x < y$, and so $xRy$. Let $y \in V \setminus[0,1)$. Then from (\ref{xleVz}) and (iv) of Lemma \ref{lemma_in_th} it follows that $x < \mathsf{m}(y) \leq \mathsf{M}(y)$, and so by definition of $R$, we get $xRy$.

    From (\ref{xaral}), (\ref{xright}), and (\ref{zleft}) it follows that $\langle X, \tau \rangle$ is $R$-bidirected along $[0, 1)$.
\end{proof}

\begin{corollary} \label{compnotim}
    If ${\it b}\mathbb{S}$ is a Hausdorff compactification of the Sorgenfrey line, then it is not a continuous open image of the Sorgenfrey line.
\end{corollary}

\begin{proof}
    By ${\it b}\mathbb{S}^*$ denote ${\it b}\mathbb{S} \setminus \mathbb{S}$. From \cite[Ch. 4, Pr. 43]{9} it follows that ${\it b}\mathbb{S}^*$ is a dense subset of ${\it b}\mathbb{S}$. Then from Theorem \ref{densecodense} it follows that ${\it b}\mathbb{S}$ is not a continuous open image of the Sorgenfrey line.
\end{proof}

\begin{lemma} \label{preservedbyclose}
    The property of being a continuous open image of the Sorgenfrey line is preserved by closed subspaces without isolated points.
\end{lemma}

\begin{proof}
    Let $\langle X, \tau \rangle$ be a topological space such that there exists a continuous open surjection $f: \mathbb{S} \rightarrow \langle X, \tau \rangle$. Let $F \subseteq X$ be a closed subset of $\langle X, \tau \rangle$ without isolated points and $Z := f^{-1}[F]$. Since $f$ is a continuous function, we have

    \begin{equation}\label{Zclosed}
        Z \text{ is a closed subset of } \mathbb{S}.
    \end{equation}

    From \cite[Ch. 2, Pr. 337]{9} it follows that $f {\upharpoonright} Z: Z_{\mathbb{S}} \rightarrow \langle F, \tau {\upharpoonright} F \rangle$ is a continuous open surjection. Since $\langle F, \tau {\upharpoonright} F \rangle$ has no isolated points and $f {\upharpoonright} Z$ is an open map, we see that

    \begin{equation}\label{Zperfect}
        Z \text{ has no isolated points}.
    \end{equation}

    From (\ref{Zclosed}), (\ref{Zperfect}) and \cite[(iii) of Theorem 4.6]{8} it follows that $Z_{\mathbb{S}} \cong \mathbb{S}$. And so $\langle F, \tau {\upharpoonright} F \rangle$ is a continuous open image of the Sorgenfrey line.
\end{proof}

\begin{corollary} \label{ifcontsnotim}
    Suppose that the Sorgenfrey line is embeddable in a Hausdorff compact space $\langle X, \tau \rangle$. Then $\langle X, \tau \rangle$ is not a continuous open image of the Sorgenfrey line.
\end{corollary}

\begin{proof}
    Let $A \subseteq X$ be such that $\langle A, \tau {\upharpoonright} A \rangle \cong \mathbb{S}$. Let $S := \mathsf{Cl}_{\langle X, \tau \rangle}(A)$, then $\langle S, \tau {\upharpoonright} S \rangle$ is a compactification of $\mathbb{S}$. So from Corollary \ref{compnotim} it follows that
    \begin{equation}\label{Snotimage}
        \langle S, \tau {\upharpoonright} S \rangle \text{ is not a continuous open image of the Sorgenfrey line}.
    \end{equation}

    From Lemma \ref{preservedbyclose} and (\ref{Snotimage}) it follows that $\langle X, \tau \rangle$ is not a continuous open image of the Sorgenfrey line.
\end{proof}

\end{document}